\documentclass[12pt, a4paper]{amsart}
\usepackage[english]{babel}
\usepackage[T1]{fontenc}
\usepackage{amsmath,amsfonts}
\usepackage{amsthm}
\usepackage{amssymb}
\usepackage[utf8]{inputenc}
\usepackage[top=4cm, bottom=4cm, left=3.5cm, right=3.5cm]{geometry}
\title
{On the rationality conjecture of some finite CW-complexes}
\usepackage{systeme}
\usepackage{pifont}
\usepackage[all]{xy}
\usepackage{tikz}
\usepackage{caption}
\usetikzlibrary{fit,calc,positioning,decorations.pathreplacing,matrix}
\newtheorem{theorem}{Theorem}[subsection]

\newtheorem{example}[theorem]{Example}
\newtheorem{lemma}[theorem]{Lemma}

\newtheorem{conjecture}[theorem]{Conjecture}

\newtheorem{proposition}[theorem]{Proposition}

\newtheorem{remark}[theorem]{Remark}

\usepackage{graphicx}
\date{\today}

\subjclass{Primary 55P10; Secondary 55M30 }

\keywords{LS-category, (Higher) Topological complexity, James generalized Hopf invariants, Zero-divisor-cup-length.}
\begin{document}
\author{A. Boudjaj and Y. Rami}{\let\thefootnote\relax\footnote{{\it Address}: My Ismail University of Meknes, Department of Mathematics  B. P. 11 201 Zitoune,  Meknès, Morocco}\let\thefootnote\relax\footnote{{\it Emails}: a.boudjaj@edu.umi.ac.ma  and y.rami@umi.ac.ma}}

\renewcommand{\abstractname}{Abstract}


\maketitle

\begin{abstract}
In this paper, we establish the rationality conjecture raised in \cite{FKS} for   any $(r-1)$-connected ($r\geq 2$) $kr$-dimensional  CW-complex  $X$ ($k\geq 2$) having a unique spherical cohomology class $u\in \tilde{H}^r(X, \mathbb{Z})$ such that $u^k\not =0$.
Next, we illustrate (topologically) our result by giving the minimal cell structure of such a CW-complex whose cohomology is a truncated polynomial algebra.
\end{abstract}

\section{Introduction and statement of the results}
Throughout this paper, unless stated otherwise, $X$ will denote  an  $(r-1)$-connected finite CW-complex $X$ ($r\geq 2$) and $\mathbb{K}$ an arbitrary  field.

Very recently, M. Farber et al. have posed in \cite{FKS}, the following rationality conjecture  
\begin{conjecture}\label{ratconj}
	The $TC$-generating function $\mathcal{F}_X(x)=\sum_{n=1}^{n=\infty}TC_{n+1}(X)x^n$ is a rational function  with a single pole of order 2 at $x = 1$. More explicitly, there exists $P\in \mathbb{Z}[X]$ such that: $$\mathcal{F}_X(x):=\sum_{n=1}^{n=\infty}TC_{n+1}(X)x^n=\frac{P(x)}{(1-x)^2}.$$
\end{conjecture}

To make precise terms of $\mathcal{F}_X(x)$, recall that the sectional category (or Schwartz's genus \cite{Sch}) of a fibration  $F\rightarrow E\stackrel{p}{\rightarrow} B$ is the minimal integer $k \geq  0$  denoted $secat(p)$ such that there exists an open cover $B = U_0\cup  U_1 \cup \ldots \cup U_k$ with the property that over each set $U_i$ the fibration admits a continuous section.  $secat(-)$ is indeed  an homotopy invariant. Special cases useful for us are the (higher) topological complexities 
  $TC_n(X)=secat(\pi_n)$ ($n\geq 2$) introduced firstly by M. Farber  in \cite{Fa03} for $n=2$ and generalized for any $n\geq 2$ by  Y. Rudyak in \cite{Ru} (see also \cite{BGRT}). Here
   $$\begin{array}{cccc}
      \pi_n : &  P(X) &  \rightarrow & X^n\\
              &    \gamma & \mapsto   & (\gamma(0), \gamma(\frac{1}{n}), \ldots , \gamma(\frac{n-1}{n}), \gamma(1)) 
              \end{array}                    
      $$ 
        is the  fibrational substitute for the iterated diagonal map $\Delta_n : X \rightarrow X^n$. In fact, $\pi_n$ describes a motion planning algorithm of a physical system moving in its configuration space $X$ between any two positions (the input and the output) but having to  reach $n-2$ intermediate stats.
        Referring to  \cite[Lemma 1]{FKS}, if $X$ satisfies (\ref{ratconj}) then for all $n$ large enough $TC_{n+1}(X) - TC_{n}(X)$ becomes  constant. 
The invariants $TC_n(X)$ are usually difficult to determine and are often approximated by  
the  algebraic invariants \textit{cup-length} and \textit{zero-divisors-cup-length} defined respectively as follows:
$$
 cup_{\mathbb{K}}(X) = sup\{j\geq 1,\; \exists a_1, a_2, \ldots , a_j\in \tilde{H}^ *(X, \mathbb{K})\; \hbox{such that} \;  a_1 a_2\ldots  a_j\not =0\}
$$ 
and
$$zcl_n(X, \mathbb{K})= sup\{j\geq 1,\; \exists a_1,\; \ldots ,\; a_j\in Ker(\smile _{n, \mathbb{K}})\; \hbox{such that} \;  a_1 \ldots  a_j\not =0\}$$
where $\smile_{n, \mathbb{K}} : H^* (X, \mathbb{K})^{n \otimes }\rightarrow H^* (X, \mathbb{K})$
denotes the graded algebra multiplication  generalizing the standard cup-product $u\smile 1_{\mathbb{K}}=\smile_{2, \mathbb{K}}$. These will be subsequently used.


Our goal in this paper is to study the conjecture (\ref{ratconj}) for any $(r-1)$-connected finite   CW-complex ($r\geq 2$) satisfying the following condition:
\begin{equation}\label{xi}
{\it \dim X=kr, \; \hbox{and}\; \tilde{H}^r(X, \mathbb{Z})\cong\mathbb{Z}u,\;  \hbox{with}\; u^k\not =0\; \hbox{some} \; k\geq 2.}
\end{equation}         
Using the Universal Coefficients Theorem and the Hurewicz Theorem we obtain 
$\tilde{H}^r(X, \mathbb{Z}) \cong Hom_{\mathbb{Z}}(\tilde{H}_r(X, \mathbb{Z}), \mathbb{Z})\stackrel{{hur^{\vee}_X}{\cong}}{\longleftarrow} Hom_{\mathbb{Z}}(\pi_r(X), \mathbb{Z}).$
It results that, if $a\in \tilde{H}_r(\mathbb{S}^r, \mathbb{Z})$ and $[f]\in \pi_r(X)\cong \mathbb{Z}$ are generating  homology and homotopy classes respectively, then,  $<f^*(u),a>=<u,f_*(a)>\not =0$ and consequently, $f^*(u)\not =0$. Thus $u\in \tilde{H}^{r}(X, \mathbb{Z})$ is indeed a {\it spherical cohomology class} \cite[Lemma 3.1]{GSW}.

To establish our main result, we have to determine the characteristics of the fields $\mathbb{K}$ for which $u^k_{\mathbb{K}}$ is non-zero. Then, we specify among the latter, those that realize equalities $TC_n(X)= zcl_n(X, \mathbb{K})$ ($\forall n\geq 2$). From this last condition, there arises the integers $\lambda_{(3,k)} = \sum_{ 0 \leq i  \leq k}  (-1)^{i}(C_k^{i})^3$ which depend only on $k$.

Based on \cite[Theorem 1]{FKS}, we establish the following 
\begin{proposition}\label{Prop1} Let $X$ be an $(r-1)$-connected finite CW-complex  satisfying (1). Then $X$ verifies (\ref{ratconj}) provided there exists a  field $\mathbb{K}$ with characteristic zero or  a prime number not dividing  $\lambda _{(3,k)}$ (resp.  $2\lambda _{(3,k-1)}$)  when  $k$ is even  (resp. when $k$ is odd) 
\end{proposition}
As a immediate consequence of Dirichly's Theorem, we state our main result as follows:
\begin{theorem} \label{Thm1}
Any  $(r-1)$-connected finite CW-complex  satisfying (1) 
	verifies the conjecture (\ref{ratconj}). 
\end{theorem} 
Using similar arguments we obtain the following (cf. Remark \ref{Rem1} below for the notion of $r$-admissible fields):
\begin{theorem}\label{Thm2}
If $X$ is an $(r-1)$-connected $kr$-dimensional  CW-complex   whose cohomology over any choice of $r$-admissible field $\mathbb{K}$ is a truncated polynomial algebra, i.e.  $H^*(X, \mathbb{K})\cong \mathbb{K}[v]/(v^{k+1})$, then $X$ verifies the conjecture (\ref{ratconj}). 
\end{theorem} 
Examples of spaces verifying  the above theorem are those satisfying in addition to (\ref{xi}) the condition: $\dim H^*(X, \mathbb{K})=k$ for every $\mathbb{K}$ with $char(\mathbb{K})\in \mathcal{P}_{(i)}$ (resp. $char(\mathbb{K})\in \mathcal{P}_{(ii)}$) (cf. \S 2 for the definitions of $\mathcal{P}_{(i)}$ and $\mathcal{P}_{(ii)}$).

In \S 2, we determine    fields $\mathbb{K}$ for which $u^k_{\mathbb{K}}\not =0$. \S 3 is devoted to proofs of our main results and in \S 4  we describe minimal structures of $(r-1)$-connected $kr$-dimensional CW-complexes such that $H^*(X, \mathbb{K})\cong \mathbb{K}[u_{\mathbb{K}}]/(u^{k+1}_{\mathbb{K}})$ for any choice of $r$-admissible  field $\mathbb{K}$. 
We end the paper by \S 5 by giving  a range of examples of   $kr$-dimensional CW-complexes of the form
 $K=\mathbb{S}^r\cup_{\beta_1} e^{2r}\cup \ldots \cup_{\beta_r} e^{kr}$ which satisfy (\ref{ratconj}).

\section{Fields satisfying $u^k_{\mathbb{K}}\not =0$} 
Throughout this section, $X$ will denote an $(r-1)$-connected finite CW complex satisfying condition (1).

Recall   that the cup-product can be defined in terms of the
Alexander-Whitney diagonal approximation:
\begin{equation}\label{AW}
\smile : H^p(X,{G_1})\otimes H^q(X,{G_2})\rightarrow H^{p+q}(X,{G_1\otimes G_2})
\end{equation}
given, in  general, for any groups $G_1$ and $G_2$  by:  
$$(a\smile b)(\sigma)=a(\sigma _{\mid [v_0, v_1, \ldots v_p]})\otimes b(\sigma _{\mid [v_p, v_{p+1}, \ldots v_{p+q}]}).$$
Therefore,  by putting in (\ref{AW})  $G_1=\mathbb{Z}$ and $G_2=\mathbb{K}$;
any  field 
and $u_{ \mathbb{K}}:=u\smile 1_{ \mathbb{K}}$, we obtain $u_{ \mathbb{K}}^k=u^k\smile 1_{ \mathbb{K}}$. 
Thereafter, we will  make use of the identification:
$$H^{kr}(X, \mathbb{Z})=\mathbb{Z}^{m_k}\oplus (\mathbb{Z}/p_{1}^{\alpha_1}\mathbb{Z})\oplus (\mathbb{Z}/p_{2}^{\alpha_2}\mathbb{Z})\oplus \ldots \oplus (\mathbb{Z}/p_{{\theta(k)}}^{\alpha_{\theta(k)}}\mathbb{Z}).$$
and will consider the family:
 $$\mathcal{P}_{u^k}:=\{p_{1}, p_{2}, \ldots , p_{{\theta(k)}}\}.$$

Next, although  $u\in \tilde{H}^r(X, \mathbb{Z})$ has infinite additive order, it may exist  $m\in \mathbb{N}^*$ such that $mu^{l+1}=0$ for  some $\; 2\leq l\leq k-1$ (i.e. in the chain level, $mu^{l+1}=\delta v$ for some $v\in C^{{(l+1)}r+1}(X)$).
Thus, we should discus two cases depending on the additive order $o(u^k)$ of $u^k$. In all what follows, for  any group $G$,  its free part will be denoted by $Free(G)$.
\begin{enumerate}
\item[(i)]
 If $o(u^k)$  is infinite so that $Free(H^{lr}(X, \mathbb{Z}))\not= 0$ for each $\; 1\leq l\leq k$,
condition (1) implies $u^k_{\mathbb{K}}\not =0$
 for all fields $\mathbb{K}\in \mathcal{P}\backslash \mathcal{Q}_{u^k}$, where 
 $\mathcal{Q}_{u^k}:=\{q_{1}, q_{2}, \ldots , q_{{\gamma (k)}}\}$ is such that  $u^k=q_{1}^{f_1} q_{2}^{f_2} \ldots  q_{{\gamma (k)}}^{f_{\gamma (k)}}w; \; w\in {\mathbb{Z}^{m_k}}$. 

\item[(ii):] If $o(u^k)$ is finite,
 so that $mu^{l+1}=0$ for  some $m\in \mathbb{N}^*$ and $\; 2\leq l\leq k-1$. Hence, $mu^k=0$ and therefore  $o(u^k)  =: p_{1}^{\beta_1}p_{2}^{\beta_2}\ldots p_{{\beta(k)}}^{\beta_{\theta(k)}}/m$ such that for any $1\leq i\leq \theta(k)$, we have  $0\leq \beta_i\leq \alpha_i$ and at least one $\beta_i>0$. Thus,
 on one hand,
for $\mathcal{Q}_{u^{l}}$ being defined similarly as $\mathcal{Q}_{u^{k}}$,  condition (1) implies that
$u^{l}_{\mathbb{K}}\not =0$
 for all   $\mathbb{K}^{'s}$ such that $char(\mathbb{K})\in \mathcal{P}\backslash \mathcal{Q}_{u^{l}}$ where $\mathcal{P}$ stands for the set of all prime numbers.
On the other one, for any field whose characteristic lies to $\mathcal{P}\backslash \mathcal{P}_{u^k}$; 
   $u^{l+1}_{\mathbb{K}}\not =0, \ldots,  u^{k}_{\mathbb{K}}\not=0$.
 \end{enumerate}
 In all what follows we put $$\mathcal{P}(i)=\{0\}\cup\mathcal{P}\backslash \mathcal{Q}_{u^{l}}\; \hbox{and}\;  \mathcal{P}(ii)=\mathcal{P}\backslash (\mathcal{P}_{u^k}\cup \mathcal{Q}_{u^{l}}).$$
 As a Consequence, we have established the following 
\begin{lemma}\label{Lem1}
$\dim Span\{u_{\mathbb{K}}, u^2_{\mathbb{K}}, \ldots , u^k_{\mathbb{K}}\}=k$ if and only if $char(\mathbb{K})\in \mathcal{P}(i)$ or else $char(\mathbb{K})\in \mathcal{P}(ii)$.
\end{lemma} 
\section{Proofs of our main results}
The following lemma will clarify the behavior of  $zcl_n(X, \mathbb{K})$ towards $char(\mathbb{K})$ and the parity of the nil-potency order $k$. 
For any $n\geq 3$ we put:
$$A_1 = u_{\mathbb{K}} \otimes 1 \otimes \dots \otimes 1, \; A_2 = 1\otimes u_{\mathbb{K}}  \otimes 1 \otimes \dots \otimes 1,\;  \dots ,\; A_n = 1 \otimes \dots \otimes 1\otimes u_{\mathbb{K}}.$$
$$\xi_{(n,k)}=[(\prod_{i=2}^{i=n}(A_1 - A_i))(A_2 - A_3)]^k$$ and $$\mu_{(n,k)}=\xi_{(n,k-1)}[(A_1 - A_2)(A_1 - A_3)\ldots (A_1-A_n)]=\xi_{(n,k)}/(A_2-A_3).$$ 
\begin{lemma}\label{Lem2}
	Let $\mathbb{K}\in \mathcal{P}(i)$ (resp. $\mathcal{P}(ii)$). 
	Then, for any $n\geq 3$:
	\begin{enumerate}
		\item[(i)] $\;\;\; \xi_{(n,2k)} = \lambda _{(3,2k)}u^{k}_{\mathbb{K}}  \otimes \dots \otimes  u^{k}_{\mathbb{K}}\;\;\;$  and $\;\;\;\xi_{(n,2k+1)} =0$.
		\item[(ii)]
		$\;\;\; \mu_{(n,2k+1)}(A_1-A_n)=2(-1)^{n-1}\lambda _{(3,2k)}u^{k}_{\mathbb{K}}  \otimes \dots \otimes  u^{k}_{\mathbb{K}}.$
		where: $$\lambda _{(3,2k)} = 2[\sum_{ 0 \leq i  \leq k}  (-1)^{i}(C_{2k}^{i})^3] +  (C_{2k}^k)^3.$$
	\end{enumerate}
\end{lemma}
\begin{proof}
	Notice first the following inductive formulas for any $n\geq 3$, $k\geq 2$ and $l\geq 2$:
	$$\xi_{(n+1,k)} = \xi_{(n,k)}(A_1-A_{n+1})^k\;\;\;\; \hbox{and}\;\;\; \xi_{(n,k+l)} = \xi_{(n,k)} \xi_{(n,l)}.$$
	Now, since $u^{j}_{\mathbb{K}}=0, \forall j>k$, a straightforward calculus gives  $$\xi_{(n,k)}= \lambda_{(n,k)} u^{k}_{\mathbb{K}}  \otimes \dots \otimes  u^{k}_{\mathbb{K}}$$
	(some constant $\lambda_{(n,k)}\in \mathbb{Z}$). Thus, using again $u^{j}_{\mathbb{K}}=0, \forall j>k$ we deduce that in the relation  $\xi_{(n+1,k)} = \xi_{(n,k)}(A_1-A_{n+1})^k$, the term $(-1)^{k}A_{n+1}^{k}$  is the only one to be retained from  teh equality $(A_1-A_{n+1})^k=\sum_{i=0}^{i=k}(-1)^{k-i}C_{k}^iA_1^iA_{n+1}^{k-i}$.
	Hence: $$\xi_{(n+1,k)}=(-1)^{k}\xi_{(n,k)}A_{n+1}^{k}=(-1)^{k}\lambda_{(n,k)} u^{k}_{\mathbb{K}}  \otimes \dots \otimes  u^{k}_{\mathbb{K}}$$ (with $n+1$ factors $u^{k}_{\mathbb{K}}$) so $\lambda_{(n+1,k)}=(-1)^{k}\lambda_{(n,k)}=\ldots =(-1)^{(n-2)k}\lambda_{(3,k)}.$
Consequently, we vave:
	\begin{equation}\label{lambda}
	\lambda_{(n,k)} = (-1)^{(n-1)k}\lambda_{(3,k)}, \forall n\geq 3.
	\end{equation}
	Now, by  using  Newton's formula for each term in  $$\xi _{(3,k)}=(-1)^k(A_1 - A_2)^k(A_3 - A_2)^k(A_2-A_3)^k,$$  we get the deterministic coefficient:
	$$\lambda_{(3,k)} = \sum_{ 0 \leq i  \leq k}  (-1)^{i}(C_k^{i})^3.$$
	Thus, 
	$$\lambda_{(n,2k)}= 2[\sum_{ 0 \leq i  \leq k}  (-1)^{i}(C_{2k}^{i})^3] +  (C_{2k}^k)^3\;\;\; \hbox{and}\;\;\; 
	\lambda_{(n,2k+1)}=0.$$

	Next, consider
	$\mu_{(n,k)}=\xi_{(n,k-1)}[(A_1 - A_2)(A_1 - A_3)\ldots (A_1-A_n)]$. Formulas lying coefficients and roots in $Q(A_1)=(A_1 - A_2)(A_1 - A_3)\ldots (A_1-A_n)$ and  the constraints $u^{j}_{\mathbb{K}}=0, \forall j>k$ lead us to take only into account the constant coefficient  $q_0=(-1)^{n-1}A_2A_3\ldots A_n$ as well as  that of $A_1$, i.e. that of $q_1=(-1)^{n-2}\sum_{2\leq i_1<\ldots <i_{n-2}\leq n}\limits A_{i_1}\ldots A_{i_{n-2}}$. Therefore,
	$$\mu_{(n,k)} =(-1)^{n-1}\lambda_{(n,k-1)}[\widehat{u^{k-1}_{\mathbb{K}}} \otimes u^{k}_{\mathbb{K}} \ldots \otimes  u^{k}_{\mathbb{K}} -\sum_{j=2}^{j=n} u^{k}_{\mathbb{K}}  \otimes\ldots \otimes  u^{k}_{\mathbb{K}}\otimes \widehat{u^{k-1}_{\mathbb{K}}}\otimes  u^{k}_{\mathbb{K}} \dots \otimes  u^{k}_{\mathbb{K}}]$$
	where $\widehat{u^{k-1}_{\mathbb{K}}}$ means that this factor is in the $j^{-th}$ place.	
	It is then immediate that $$\;\;\; \mu_{(n,2k+1)}(A_1-A_n)=2(-1)^{n-1}\lambda _{(3,2k)}u^{k}_{\mathbb{K}}  \otimes \dots \otimes  u^{k}_{\mathbb{K}}.$$
\end{proof}

\subsection{Proof of Proposition \ref{Prop1}}
The following theorem will be subsequently used:
\begin{theorem} \cite{CLOT, BGRT}: For  any $s$-connected CW-complex $X$, 
	 any field $\mathbb{K}$ and any integer $n\geq 2$  we have
	$$cup_{\mathbb{K}}(X)\leq cat_{LS}(X)\leq \frac{\dim X}{s+1}\; \hbox{and}\; zcl_n(X, {\mathbb{K}}) \leq TC_n(X)\leq \frac{n\dim X}{s+1}.$$
\end{theorem}  
\begin{proof}
By the above theorem, $(r-1)$-connectedness with $r\geq 2$ and dimension $kr$ of $X$ with $k\geq 2$ imply that $cat_{LS}(X)\leq k$ and $TC_{n}(X)\leq nk$ for every $n\geq 2$. We continue considering the following steps (we fix $\mathbb{K}$ a field satisfying hypothesis of Proposition \ref{Prop1}):
\begin{enumerate}
\item By Lemma \ref{Lem1} we have  $u^k_{\mathbb{K}}\not =0$, hence, $k\leq cup_{\mathbb{K}}(X)\leq cat_{LS}(X)$ \cite[Proposition 1.5]{CLOT} and so $cup_{\mathbb{K}}(X)=k= cat_{LS}(X)$. 
\item  If $n=2$ one shows easily (for any $\mathbb{K}$) that  $\xi _{(2, k)}=(A_1-A_2)^{2k}\not =0$, thus, $2k\leq zcl_2(X, {\mathbb{K}})$ and therefore   $zcl_2(X, {\mathbb{K}})=2k=TC_2(X)$.
\item Assume for the rest  that $n\geq 3$:
	
	(i) if $k$ is even, $k\geq 2$, using  Lemma \ref{Lem2}, we see that 
	$$\xi _{(n, k)}\not =0\Leftrightarrow \lambda _{(n, k)}\not =0\Leftrightarrow \lambda_{(n, k)}\wedge char(\mathbb{K})=1$$
	in which case we have $nk\leq zcl_n(X, {\mathbb{K}})$.
	 Therefore,   $zcl_n(X, {\mathbb{K}})=nk=TC_n(X)$.

	(ii) if $k$ is odd, the same argument using $\mu _{(n,k)}(A_1-A_n)$ gives us the same equality $zcl_n(X, {\mathbb{K}})=nk=TC_n(X)$.  
\end{enumerate}
To conclude, it remain to use \cite[Theorem 1]{FKS}.
\end{proof}
\subsection{Proof of Theorem \ref{Thm1}}
\begin{proof}
	Proposition \ref{Prop1}  is satisfied for any field whose characteristic
	  lies to  $\mathcal{P}_{(i)}$ or else $\mathcal{P}_{(ii)}$ and does not divide the integer $\lambda _{(3,k)}$ or $\lambda _{(3,k-1)}$ depending on the parity of $k$.
	It suffice then to make use of Dirichely's theorem  by which we may  fix such a field.
\end{proof}

Next, to establish Theorem \ref{Thm2}  recall that a  field $\mathbb{K}$ is said {\it $r$-admissible}  if there exists an $(r-1)$-connected CW-complex $X$ such that $H^*(X, \mathbb{K})\cong \mathbb{K}[v]/(v^{k+1})$. Referring to \cite[Theorem 2]{To63} and \cite[Theorem 4L.9 and the remark just after]{Ha}  we note on one hand that  if $char(\mathbb{K})=p$, a prim number, then  the nil-potency order $k$ of the indeterminate $v$ should be greater than $p-1$ and on the other we resume in the following remark complete results concerning this notion: 
\begin{remark}\label{Rem1}  
(a):	Let $\mathbb{K}$ be a  field of characteristic zero or a prim $p$. Then, $\mathbb{K}$ is $r$-admissible
\begin{enumerate}
\item  for every even $r\geq 2$ if characteristic is zero.
\item
for $r\in \{2,\; 4\}$   if  $p\geq 2$.
\item for $r=8$ if 
 $p=2$ or an arbitrary odd prim of the form $p=4s+1$ some $s\geq 1$. 
\item for every even $r\notin  \{2, 4, 8\}$ if $p$ is  an arbitrary odd prim of the form
 $p=\frac{r}{2}s+1$ some $s\geq 1$.
\end{enumerate}	
(b): According to Dirichly's theorem,  there are infinitely
many primes in any  arithmetic progression $p=\frac{r}{2}s+1$, that is, 
 to each  integer $r\geq 2$ it is associated infinitely many  $r$-admissible  fields.
\end{remark} 
\subsubsection{\bf Proof of Theorem \ref{Thm2}}
Notice first that with hypothesis of Theorem \ref{Thm2} and Remark \ref{Rem1} above, Lemma \ref{Lem1} and Lemma \ref{Lem2} are  satisfied for any choice of $r$-admissible field $\mathbb{K}$ when we  replace $v$ by $u_{\mathbb{K}}$. 
\begin{proof}
 It suffice to use  successively similar arguments as in proves of Proposition \ref{Prop1} and Theorem \ref{Thm1}.
\end{proof}

\section{Minimal cell structure of CW-complexes with truncated polynomial cohomology}
Our goal in this section is to give a " simultaneously "  example of spaces $X$ satisfying (\ref{xi}) and having  a  truncated polynomial algebra cohomology over any choice of $r$-admissible  field $\mathbb{K}$ i.e.  $H^*(X, \mathbb{K})\cong \mathbb{K}[u_{\mathbb{K}}]/(u^{k+1}_{\mathbb{K}})$. These (as it is mentioned in the introduction) form a particular  class among those given by Theorem \ref{Thm2}. 
Simple examples of such spaces are  $\mathbb{C}P^k$ and  $\mathbb{H}P^k$ \cite[Corollary 4L. 10]{Ha}. Notice however that real projective spaces $\mathbb{R}P^{2k}$ and  $\mathbb{R}P^{2k+1}$  do not satisfy (1) since  $H^*(\mathbb{R}P^{2k}, \mathbb{Z})\cong \mathbb{Z}[\alpha]/(2\alpha, \alpha ^{k+1})$;  $|\alpha|=2$ and 
$H^*(\mathbb{R}P^{2k+1}, \mathbb{Z})\cong \mathbb{Z}[\alpha, \beta]/(2\alpha, \alpha ^{k+1}, \beta ^2, \alpha \beta)$;  $|\alpha|=2$ and $|\beta|=2k+1$ \cite{Ha}. 

More examples are among 
CW-complexes of the form  $L=\mathbb{S}^{n_1}\cup_{\beta_2} e^{n_2}\cup \ldots \cup_{\beta_{k-1}} e^{n_{k-1}}$, with $n_{i+1}-1\geq n_i\geq 2$. Indeed, I. M. James
 associated to  
any $\beta _k \in \pi_{n_k-1}(L)$ such that $n_k=n_{k-1}+n_1$ its {\it generalized Hopf invariant} $h_k([\beta_k])$  to be the integer $m_k$ satisfying   $x_1x_{{k-1}}=m_kx_{k}$. Here  $x_i$ stands for  the generating element in $H^{n_i}(L\cup_{\beta _k} e^{n_k}, \mathbb{Z})$. He then showed  that this correspondence defines an homomorphism $h_k: \pi_{n_k-1}(L)\rightarrow \mathbb{Z}$.
As a particular case, if $n_1=r\geq 2$ and  $n_{i+1}=n_i+n_1$ for each  $1\leq i \leq k-1$, we get $ K'=\mathbb{S}^r\cup_{\beta _2} e^{2r}\cup \ldots \cup_{\beta_{k-1}} e^{(k-1)r}$ and $h_k: \pi_{rk-1}(K')\rightarrow \mathbb{Z}$ is an homomorphism  extending the usual Hopf invariant $H=:h_2: \pi_{2r-1}(\mathbb{S}^r)\rightarrow \mathbb{Z}$. Thereafter, we denote $K=\mathbb{S}^r\cup_{\beta _2} e^{2r}\cup \ldots \cup_{\beta_r} e^{kr}$  obtained from $K'$ and  $\beta_k$ via the relation $x_1x_{k-1}=h_k([\beta_k])x_{k}=: m_kx_k$ ($k\geq 3$).
Now, by using the long cohomology exact sequences of pairs $$(\mathbb{S}^r\cup_{\beta _2} e^{2r}\cup \ldots \cup_{\beta_{i}} e^{ir} \; ,\;  \mathbb{S}^r\cup_{\beta _2} e^{2r}\cup \ldots \cup_{\beta_{i+1}} e^{(i+1)r})$$ we exhibit a family of isomorphisms
$$H^{ir}(\mathbb{S}^r\cup_{\beta _2} e^{2r}\cup \ldots \cup_{\beta_{i+1}} e^{(i+1)r})\cong H^{ir}(\mathbb{S}^r\cup_{\beta _2} e^{2r}\cup \ldots \cup_{\beta_{i}} e^{ir})$$ 
($2\leq i \leq k-1$).  Knowing   that $H^{r}(\mathbb{S}^r\cup_{\beta _2} e^{2r})\cong H^{r}(\mathbb{S}^r)\cong \mathbb{Z}x_1$  we may    identify   each $x_i$ with its antecedent, hence,  if we put $H([\beta_2]) = h_2([\beta_2]) =: m_{2}$, we see by induction that :
$$ m_2\ldots m_{k-1}m_{k}x_k=x_1^k.$$
Consequently,  every  CW-complex
$K=\mathbb{S}^r\cup_{\beta _2} e^{2r}\cup \ldots \cup_{\beta_r} e^{kr}$ whose attaching maps are such that  $m_2\ldots m_{k-1}m_{k}\not =0$ satisfies  (1) and  $o(x_1^k)$ is infinite since the one on $x_k$ is (due to the hypothesis $r\geq 2$). Moreover,  (cf. \S 2), $\mathcal{Q}_{x_1^k}=\{q_i\in \mathcal{P},\; : \; q_i / m_2\ldots m_{k-1}m_{k}\} \; \hbox{and}\; \mathcal{P}_{(i)}=\mathcal{P}\backslash \mathcal{Q}_{x_1^k}.$

Next, put $L=K\cup _{\beta '_{l+1}} e^{(l+1)r+1}$ (some $l<k$) such that (in cellular homology) the cell $e^{(l+1)r+1}$ satisfies   
$$\delta (e^{(l+1)r+1}) = p_{1}^{\alpha _1}p_{2}^{\alpha _2} \ldots  p_{{\theta(k)}}^{\alpha _{\theta(k)}}e^{(l+1)r}.$$
It is clear that $L$ satisfies 
(1) 
 with $x_1^{l+1}$ and consequently $x_1^k$  of finite order, and  
$\mathcal{Q}_{x_1^{l}}=\{q_i\in \mathcal{P},\; : \; q_i / m_2\ldots m_{l-1}m_{l}\} \; \hbox{and}\; \mathcal{P}_{(ii)}=\mathcal{P}\backslash (\mathcal{P}_{x_1^k}\cup \mathcal{Q}_{x_1^l}).$

Our main result in this section is given with the following terms already used in \cite[Proposition 4C.1]{Ha} :\\
\textit{$\tilde{e}^n$ will denote the pair of cells $(e^{n+1}, e^n, \delta (e^{n+1})=q^te^n)$  where $q^t$ is the additive order of the homology (hence the cohomology) class induced by a summand $\mathbb{Z}/q\mathbb{Z}$.} 
\begin{theorem}\label{Thm3}
	Let  $X$ be a  finite  $kr$-dimensional CW-complex satisfying (1) with cohomology a truncated polynomial algebra  $H^*(X, \mathbb{K})\cong \mathbb{K}[u_{\mathbb{K}}]/(u^{k+1}_{\mathbb{K}})$
	for any choice of  $r$-admissible  field  $\mathbb{K}$.  
	\begin{itemize}
		\item[(a)] If $o(u^k)$ is infinite, then $X$ is homotopy equivalent 
		to a CW-complex\\
		$K= \mathbb{S}^r\cup _{\eta _2}e^{2r}\cup   \ldots  \cup _{\eta _k} e^{kr}$
		such that  
		$h_2(\eta _2)h_3(\eta _3)\ldots h_k(\eta _k) = q_{1}^{f_1} q_{2}^{f_2} \ldots  q_{{\gamma (k)}}^{f_{\gamma (k)}}$
		where $q_i\in \mathcal{Q}_{u^k}$.
		\item[(b)] If $o(u^k)$ is finite  and $l\geq 2$ is the greatest power of $u$ with infinite order, 
		then $X$ is homotopy equivalent to a CW-complex $K=L\cup _{\varrho} L'$ where\\
		- $L= \mathbb{S}^r\cup _{\eta _2}e^{2r}\cup   \ldots  \cup _{\eta _l} e^{kl}$ 
		such that  
		$h_2(\eta _2)h_3(\eta _3)\ldots h_l(\eta _l) = q_{i_1}^{g_{i_1}} q_{i_2}^{g_{i_2}} \ldots  q_{i_{\gamma(l)}}^{g_{i_{\gamma (l)}}}$
		with $q_{i_j}\in \mathcal{Q}_{u^{l}}$ and $0\leq g_{i_{j}}\leq f_{i_{j}}$ and\\
		- $L'$ is formed, for each $p_i\in \mathcal{P}_{u^{k}}$, by $k-l$ cells:  an  $\tilde{e}_i^{jr}$ or else an    $\tilde{e}_i^{jr-1}$ for each $l+1\leq j \leq k-1$,  each of these cells is of finite additive order ${p_i}^{\alpha'_i}$ (${\alpha'_i}\geq {\alpha_i}$) and an  $\tilde{e}_i^{kr-1}$  of finite additive order ${p_i}^{\alpha_i}$.
	\end{itemize}
\end{theorem}  
\begin{proof}
We note at the beginning that, by Lemma \ref{Lem1}, every $r$-admissible field is indeed in $\mathcal{P}_{(i)}$ or else in $\mathcal{P}_{(ii)}$. This explains notations in the statement of the theorem. 
	\begin{itemize}
		\item[(a)] Assume that $o(u^k)$ is infinite.  By hypothesis on the cohomology of $X$ the UCT 
		with $\mathbb{Q}$-coefficients, condition (1) and  Remark \ref{Rem1} imply that $H_{ir}(X, \mathbb{Z})\cong \mathbb{Z}$ for each $0\leq i\leq k$. 
		Therefor $X$ is homotopy equivalent 
				to the CW-complex
				$K= \mathbb{S}^r\cup _{\eta _2}e^{2r}\cup   \ldots  \cup _{\eta _k} e^{kr}$
				whose attaching maps satisfy the condition
				$h_2(\eta _2)h_3(\eta _3)\ldots h_k(\eta _k) = q_{1}^{f_1} q_{2}^{f_2} \ldots  q_{{\gamma (k)}}^{f_{\gamma (k)}}$
				where $q_i\in \mathcal{Q}_{u^k}$.
		\item[(b)] Assume that $o(u^k)$ is finite  and let  $l\geq 2$ be the greatest power of $u$ with infinite order.
		First of all, since the powers $u^i, \; 1\leq i\leq l$ are in the free part of $\tilde{H}^*(X, \mathbb{Z})$, we should introduce cells $e^0,\; e^r,\; \ldots ,\; e^{lr}$ in the minimal structure of $X$. These form $L= \mathbb{S}^r\cup _{\eta _2}e^{2r}\cup   \ldots  \cup _{\eta _l} e^{kl}$ 
				such that  
				$h_2(\eta _2)h_3(\eta _3)\ldots h_l(\eta _l) = q_{i_1}^{g_{i_1}} q_{i_2}^{g_{i_2}} \ldots  q_{i_{\gamma(l)}}^{g_{i_{\gamma (l)}}}$		with $q_{i_j}\in \mathcal{Q}_{u^{l}}$ and $0\leq g_{i_{j}}\leq f_{i_{j}}$. 
				
		Now,  by Lemma \ref{Lem1}, in order to realize hypothesis on the cohomology of $X$,     the cells
		   $e^{{(l+1)}r},\; \ldots ,\; e^{kr}$ should  be introduced for every   $p=char(\mathbb{K})\in \mathcal{P}_{(ii)}=\mathcal{P}\backslash (\mathcal{P}_{u^k}\cup \mathcal{Q}_{u^{l}})$.  
		Hence,  in order  to rich the (full dimension) $\dim\tilde{H}^*(X, \mathbb{K})=k$, the UCT:  
		$$H^{ir}(X, \mathbb{K})\cong Hom(H_{ir}(X, \mathbb{Z}) , \mathbb{K})\oplus Ext(H_{ir-1}(X, \mathbb{Z}), \mathbb{K})$$ imposes to add (for each $p_i\in \mathcal{P}_{(ii)}$):\\
		- $k-l-1$-generator cells $e^{{jr}}_i$ or else $e^{{jr-1}}_i$ ($l+1\leq j\leq k$) and its corresponding relater-cells $e^{{jr+1}}_i$ or else $e^{{jr}}_i$ satisfying $\delta (e^{{jr+1}}_i) = {p_i}^{\alpha'_i}e^{jr}_i$ or else $\delta (e^{{jr}}_i) = {p_i}^{\alpha'_i}e^{jr-1}_i$ with $\alpha '\geq \alpha$.\\
		-  a cell of degree $rk$ which should came from the Ext term, that is from a pair  ($e^{rk},\; e^{rk-1}$) such that $\delta (e^{{jr}}_i) = {p_i}^{\alpha_i}e^{jr-1}_i$. We then obtain $L'$ as it is stated in the theorem.
		
	\end{itemize}
\end{proof}
\section{Examples and a final remark}
To have an idea on sequences of prime numbers introduced in  Proposition \ref{Prop1}, we give in the  table below (\textbf{Table 1}) all prime factors of $\lambda _{(3, k)}$ for $2\leq k \leq 40$. This in possible thanks to the
{\bf  Scientific Workplace 5.5 software}:

	\begin{table}[ht]
			\caption{List of prime factors of $\lambda _{(3, k)}$ up to $40$}
		\begin{tabular}{|c|c|}
			\hline
			k & prime factors of $\lambda_{(3,k)}$\\
			\hline
			2 & 2, 3 \\
			\hline
			4 & 2, 3, 5\\
			\hline
			6 & 2, 3, 5, 7\\
			\hline
			8 & 2, 3, 5, 7, 11\\
			\hline 
			10 & 2, 3, 7, 11, 13\\
			\hline 
			12 & 2, 3, 7, 11, 13, 17 \\
			\hline 
			14 & 2, 3, 5, 11, 13, 17, 19\\
			\hline 
			16 & 2, 3, 5, 11, 13, 17, 19, 23\\
			\hline
			18 & 2, 3, 5, 11, 13, 17, 19, 23\\
			\hline
			20 & 2, 3, 5, 7, 11, 13, 17, 19, 23, 29 \\
			\hline
			22 & 2, 3, 5, 7, 13, 17, 19, 23, 29, 31\\
			\hline
			24 & 2, 3, 5, 7, 13, 17, 19, 23, 29, 31\\
			\hline
			26 & 2, 3, 5, 7, 17, 19, 23, 29, 31, 37\\
			\hline
			28 & 2, 3, 5, 17, 19, 23, 29, 31, 37, 41\\
			\hline
			30 & 2, 3, 5, 11, 17, 19, 23, 29, 31, 37, 41, 43\\
			\hline
			32 & 2, 3, 5, 7, 11, 17, 19, 23, 29, 31, 37, 41, 43, 47\\
			\hline
			34 & 2, 3, 5, 7, 11, 19, 23, 29, 31, 37, 41, 43, 47\\
			\hline
			36 & 2, 3, 5, 7, 11, 13, 19, 23, 29, 31, 37, 41, 43, 47, 53\\
			\hline
			38 & 2, 3, 5, 7, 11, 13, 23, 29, 31, 37, 41, 43, 47, 53\\
			\hline
			40 & 2, 3, 5, 7, 11, 13, 23, 29, 31, 37, 41, 43, 47, 53, 59\\
			\hline
			
		\end{tabular} \\

\end{table}
\subsection{Examples}
The following theorem that we recall in our notations ($h_i^r$ is denoted above by $h_i$!) is paramount to get examples of CW-complexes of the form $K= \mathbb{S}^r\cup _{\beta _2}e^{2r}\cup _{\beta _3}   \ldots  \cup _{\beta _k} e^{kr}$ satisfying the conjecture (\ref{ratconj}):
\begin{theorem}\label{Thm4} \cite[Thorem 1]{Bau}
	$$Im(h_i^r)=\Bigg\{\begin{array}{lr}
	\mathbb{Z}, & \hbox{if} \;  r=2,\; 4,\, 8 \; \hbox{and}\; i=2\\
	\mathbb{Z}, &  \hbox{if} \; r=2\; \hbox{and i a prime number}\\
	i\mathbb{Z}, & \hbox{otherwise}
	\end{array}   
	\Bigg.$$
\end{theorem}
\begin{example}
\begin{enumerate}
\item For $r=2$ and $k=2$, the space $K=\mathbb{S}^2\cup_{\beta _2} e^{4}$, where $\beta _2 : \mathbb{S}^3\rightarrow \mathbb{S}^2$ is  the Hopf map, that is such that   $h_2(\beta _2) = 1$, clearly satisfies Proposition \ref{Prop1} for any  field of characteristic $p \notin \{ 2,3 \}$. 
\item For $r=2$ and $k=3$, let $\beta_3 =[i_2, [i_2,i_2]]: \mathbb{S}^{5}\rightarrow \mathbb{S}^{2}\cup_{\beta _2} e^{4}$ satisfying $h_3(\beta_3) = 3$. The space $$K=\mathbb{S}^2\cup_{\beta _2} e^{4}\cup_{\beta_3} e^{6}$$ satisfies hypothesis of Proposition \ref{Prop1} for any  field of characteristic $p \notin \{ 2,3 \}$. 
	\item  For  $r = 4$ and $k = 3$, let $\beta _2 : \mathbb{S}^7\rightarrow \mathbb{S}^4$ be the Hopf map, that is such that   $h_2(\beta _2) = 1$, and  $\beta_3 =[i_4, [i_4,i_4]]: \mathbb{S}^{11}\rightarrow \mathbb{S}^{4}\cup_{\beta _2} e^{8}$ satisfying $h_3(\beta_3) = 3$. The space $$K=\mathbb{S}^4\cup_{\beta _2} e^{8}\cup_{\beta_3} e^{12}$$ satisfies hypothesis of Proposition \ref{Prop1} for any  field of characteristic $p \notin \{ 2,3 \}$. 
	\item  For  $r = 4$ and $k = 4$, in addition to $\beta _2$ and $\beta _3$ of the above item, consider $\beta_4 : \mathbb{S}^{15}\rightarrow \mathbb{S}^4\cup_{\beta _2} e^{8}\cup_{\beta_3} e^{12}$, provided by Theorem \ref{Thm4}, such that
	$h_4(\beta_4) = 4$. Hence, the space $$K=\mathbb{S}^4\cup_{\beta _2} e^{8}\cup_{\beta_3} e^{12} \cup_{\beta_4} e^{16}$$ satisfies hypothesis of {Proposition \ref{Prop1}.} for any  field of characteristic $p \notin \{ 2,3,5 \}$.
	\item For  $r = 8$ and $k = 7$, let $\beta_2 : \mathbb{S}^{15} \rightarrow S^8$ be the Hopf map, that is with     $h_2(\beta _2) = 1$ and consider the maps $\beta_3 : \mathbb{S}^{23}  \longrightarrow S^8 \cup_2 e^{16}$, $ \beta_4: \mathbb{S}^{31}  \rightarrow \mathbb{S}^8 \cup_2 e^{16} \cup_3 e^{24}$, $ \beta_5 : \mathbb{S}^{39}  \rightarrow \mathbb{S}^8 \cup_2 e^{16} \cup_3 e^{24} \cup_{4} e^{32}$, $ \beta_6 : \mathbb{S}^{47}  \rightarrow \mathbb{S}^8 \cup_2 e^{16} \cup_3 e^{24} \cup_{4} e^{32} \cup_5 e^{40}$ and $ \beta_7 : \mathbb{S}^{55}  \rightarrow \mathbb{S}^8 \cup_2 e^{16} \cup_3 e^{24} \cup_{4} e^{32} \cup_5 e^{40} \cup_6 e^{48}$, provided by Theorem \ref{Thm4}, such that  $h_3(\beta_3) = 3$, $h_4(\beta_4) = 4$, $h_5(\beta_5) = 5$, $h_6(\beta_6) = 6$ and $h_7(\beta_7) = 7$. Hence, the space $$K=\mathbb{S}^8\cup_{\beta _2} e^{16}\cup_{\beta_3} e^{24} \cup_{\beta_4} e^{32} \cup_{\beta_5} e^{40} \cup_{\beta_6} e^{48} \cup_{\beta_7} e^{56}$$ satisfies hypothesis of {Proposition \ref{Prop1}} for any  field of characteristic $p \notin \{ 2,3,5,7 \}$.
	\item  For $r = 8$ and $k = 8$ we consider  $\beta _2$, $\beta_3 $, $\beta_4 $, $\beta_5$, $\beta_6$,  $\beta_7$ given in the previous item and  $\beta_8 : \mathbb{S}^{63} \rightarrow \mathbb{S}^8\cup_{\beta _2} e^{16}\cup_{\beta_3} e^{24} \cup_{\beta_4} e^{32} \cup_{\beta_5} e^{40} \cup_{\beta_6} e^{48} \cup_{\beta_7} e^{56}$ such that $h_8(\beta_8) = 8$. Hence, the space $$K=\mathbb{S}^8\cup_{\beta _2} e^{16}\cup_{\beta_3} e^{24} \cup_{\beta_4} e^{32} \cup_{\beta_5} e^{40} \cup_{\beta_6} e^{48} \cup_{\beta_7} e^{56}  \cup_{\beta_8} e^{64} $$ satisfies hypothesis of {Proposition \ref{Prop1}} for any  field of characteristic $p \notin \{ 2,3,5,7,11 \}$.
\item
The above examples 	correspond to cases where $r= 2, 4, 8$.
	Now, by Theorem \ref{Thm4} when $r \not = 2,\; 4, \; 8$ we obtain a CW-complex $K$ such that the product of its generalized Hopf invariants is  $h_2(\beta _2)\ldots h_k(\beta _k)=k!$. In order to illustrates the impact of $\lambda_{(3,k)}$ in the process, we finish by fixing $r=6$ and taking some arbitrary values  of $k$. The table below dresses  fields characteristics to exclude when  $$K=\mathbb{S}^6\cup_{\beta _2} e^{12}\cup_{\beta_3} \ldots \cup_{\beta_k} e^{6k}:$$
	
	
		\begin{table}[ht]
						\caption{Finite characteristic $p$ not allowed, for some values of $k$ ($r=6$)} 
			\begin{tabular}{|c|c|}
				\hline 
				k & finite characteristic $p$ not allowed \\
				\hline
				4 & 2, 3, 5 \\
				\hline 
				5 &  2, 3, 5\\
				\hline
				16 & 2, 3, 5, 7, 11, 13, 17, 19, 23\\
				\hline 
				18 &  2, 3, 5, 7, 11, 13, 17, 19, 23\\
				\hline 
				20 &  2, 3, 5, 7, 11, 13, 17, 19, 23, 29\\
				\hline 
				22 &  2, 3, 5, 7, 11, 13, 17, 19, 23, 29, 31\\
				\hline 
			\end{tabular}\\

		
	\end{table}
	
\end{enumerate}

\end{example}

\subsection{Final remark}
	
	In this paper, we have treated  the case where $u\in \tilde{H}^r(X, \mathbb{Z})$ in condition (1) is the unique spherical cohomology class (thus of infinite order). 
	  To achieve our approach in studying the rationality conjecture for finite CW-complexes $X$ satisfying the relations $zcl_n(X)=TC_n(X)\; \forall n\geq 2$,  it remains, on the one hand,  to deal with the case where $u\in \tilde{H}^r(X, \mathbb{Z})$ is of finite order (e.g. $\mathbb{R}P^n$) and, on the other hand, to consider more than one generating cohomology class in $\tilde{H}^r(X, \mathbb{Z})$ spherical or not. This is our main project for the future.
	

\end{document}